\documentclass[11pt]{amsart}
\usepackage{graphicx}
\usepackage{pinlabel}

\theoremstyle{plain}
\newtheorem{thm}{Theorem}[section]
 \newtheorem{cor}[thm]{Corollary}
 
 \newtheorem{lem}[thm]{Lemma}
 \newtheorem{prop}[thm]{Proposition}

\newtheorem{problem}[thm]{Problem}

\newtheorem{rem}[thm]{Remark}

\begin{document}

\title[Boundary slopes of alternating knots]{Non-integral boundary slopes of alternating knots}

\author{Masaharu Ishikawa}
\address[M. Ishikawa]{Department of Mathematics, Hiyoshi Campus, Keio University, 4-1-1 Hiyoshi, Kohoku-ku, Yokohama-shi, Kanagawa, 223-8521, Japan}
\email{ishikawa@keio.jp}
\author{Thomas W.~Mattman}
\address[T. W. Mattman]{Department of Mathematics and Statistics,
California State University, Chico,
Chico, CA 95929-0525}
\email{TMattman@CSUChico.edu}
\author{Koya Shimokawa} 
\address[K. Shimokawa]{Department of Mathematics, Ochanomizu University, 2-1-1 Otsuka, Bunkyo-ku, Tokyo 112-8610, Japan}
\email{shimokawa.koya@ocha.ac.jp}

\begin{abstract}
We show, for every positive integer $n$, there is an alternating knot having a boundary slope with denominator $n$. 
We make use of Kabaya's method for boundary slopes and the layered solid torus construction introduced
by Jaco and Rubinstein and further developed by Howie et al.\
\end{abstract}

\maketitle

\section{Introduction}

Hatcher and Thurston~\cite{HT} showed every boundary slope of a
$2$-bridge knot is an even integer.
Hatcher and Oertel~\cite{HO} demonstrated that, while every rational number can be
a boundary slope of a Montesinos knot,
an alternating Montesinos knot has only even integral boundary slopes.
They asked if this might hold more generally: Is it true that every alternating knot has only even integral boundary slopes?
Kabaya \cite{Kabaya} discovered alternating knots (including $10_{82}$) having odd integral boundary slopes
and even non-integral boundary slopes (e.g., $10_{79}$),
by applying his method for the deformation variety  \cite{Kabaya-JKTR}.
In \cite{DG}, Dunfield and Garoufalidis produced many additional examples of
non-integral boundary slopes for alternating knots. Their examples include slopes with denominators 2, 3, and 5. 

In this paper, we  generate non-integral rational boundary slopes for two infinite families of alternating knots, see Figure~\ref{fig:Kn}.

\begin{figure}
\includegraphics[width=0.4\textwidth]{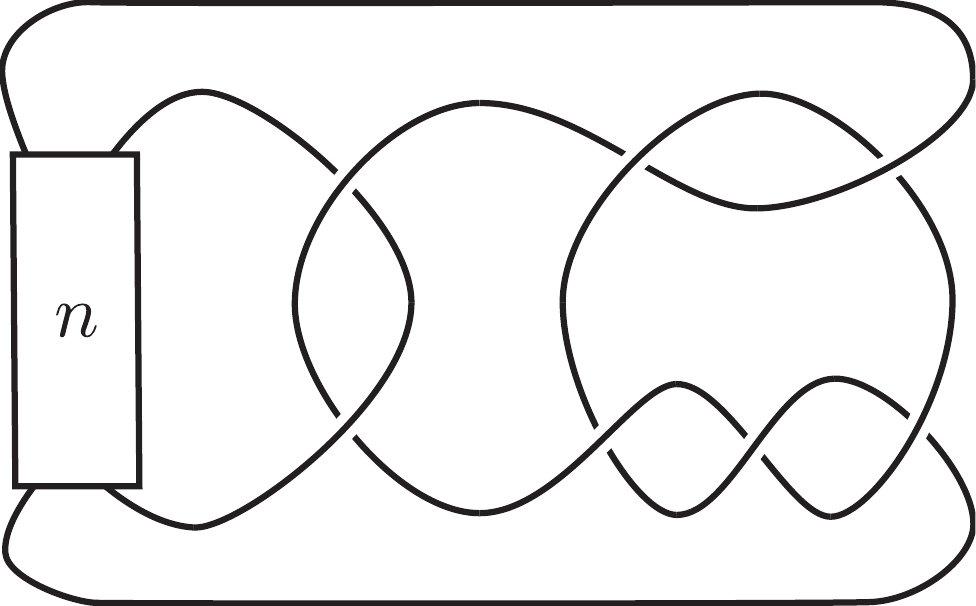}
\quad
\includegraphics[width=0.4\textwidth]{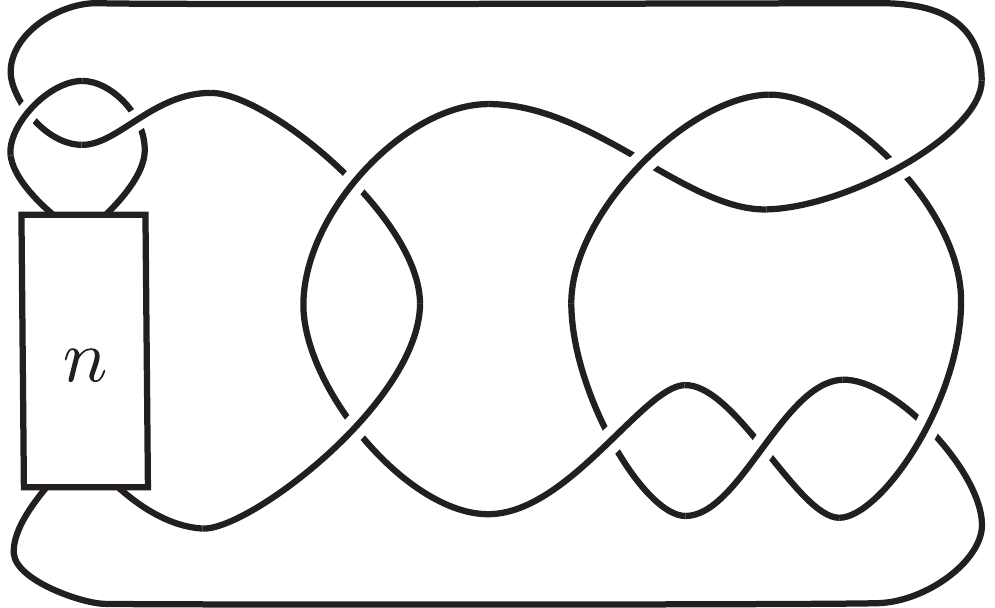}
\caption{
\label{fig:Kn}
The alternating knots $K_n$ (left) and $J_n$ (right): $K_3=10_{79}$, $K_5=12a815$, $K_7=14a12383$, $J_2 = 11a49$, $J_3 = 12a344$,
$J_4 = 13a1413$.
Here $n$ indicates the number of  half twists.}
\end{figure}

\begin{thm} 
\label{thm:Kn}
Let $n \geq 3$ be odd. The alternating knot $K_n$ has 
boundary slope  $\frac{2n^2 - 10n + 2}{n}$.
\end{thm} 

\begin{thm} 
\label{thm:Jn}
Let $n \geq 2$ be even. The alternating knot $J_n$ has 
boundary slopes  $- \frac{14n-2}{n}$ and $- \frac{10n+8}{n+1}$.
\end{thm}

As a corollary, we see that every integer is realized as a denominator.

\begin{cor}\label{cor:main}
\label{cor:non-integral}
For any positive integer $k$, there exists an alternating knot having a boundary slope with denominator $k$.
\end{cor}

The two knots that start our sequences were already known to have non-integral boundary slopes.
The knot $K_3$ is $10_{79}$, which was shown to have non-integral boundary slopes $\pm \frac{10}{3}$ by Kabaya~\cite{Kabaya} 
and Dunfield and Garoufalidis~\cite{DG}, and $J_2 = 11a49$ has slope $ - \frac{28}{3}$, again by \cite{DG}.
In~\cite{DG}, the authors give many examples of alternating knots with 10 to 12 crossings having non-integral boundary slopes.
It seems likely that many further examples could be generated by adding twists
in a twist region of one of the knots listed by Dunfield and Garoufalidis. For example, although we do not prove it here, for $n$ odd, 
$-\frac{14n+12}{n+1}$ and $ -\frac{18n-2}{n}$ are boundary slopes of $J_n$. Indeed, in \cite{DG}, the authors show that $J_3 = 12a344$ has boundary slopes
$-\frac{52}{3}$ and $-\frac{27}{2}$.

The knots in the $K_n$ family are three braids. We remark that, in unpublished work, Matsuda~\cite{Matsuda} showed that 
every rational number occurs as the boundary slope of a three-braid knot.

Our method is inspired by the layered solid torus technique for calculating $A$-polynomials 
described by Howie et al.\ in \cite{HMP, HMPT}. 
We construct explicit triangulations of the complements of
$K_n$ and $J_n$ so that adding a tetrahedron to $K_n$'s triangulation gives that of $K_{n+2}$
and similarly for the $J_n$ sequence.

This paper presents further evidence that 
boundary slopes of alternating knots behave in unexpected ways. Earlier we
\cite{IMNS} showed that the diameter of the set of boundary slopes can become arbitrarily 
large compared to the crossing number. 
Here, we have realized every denominator, but the growth of denominators found in this paper is more controlled. 
The denominators $n$ and $n+1$ (ignoring possible cancelation with the numerator) 
that appear in Theorems~\ref{thm:Kn} and \ref{thm:Jn} are linear in the crossing number:
$\mbox{cr}(K_n) = n+7$ and $\mbox{cr}(J_n) = n+9$. There is also a linear relationship with
the number of tetrahedra in the triangulations used in our argument: $\frac{n+25}{2}$ for $K_n$
and $\frac{n+34}{2}$ for $J_n$.

We close this section with the following problem.

\begin{problem}
Given a rational number $\frac{\,p\,}{q}$,
is there an alternating knot with a boundary slope $\frac{\,p\,}{q}$?
\end{problem}

In the next section we review Kabaya's method for calculating boundary slopes. In Sections 3 and 4,
respectively, we prove Theorems~\ref{thm:Kn} and \ref{thm:Jn}.

\section{Kabaya's method
\label{sec:Kab}%
}

In this section, we recall Kabaya's method \cite{Kabaya-JKTR}.
He used this method to obtain ideal points of the complement of a pretzel knot in \cite{AGT}.

Let $N$ be an orientable complete hyperbolic 3-manifold of finite volume with cusps.
We consider an ideal triangulation of $N$ with $n$ ideal tetrahedra with corresponding complex parameters $z_\nu (1\le \nu \le n)$.
Let $z'_\nu=\frac{1}{1-z_\nu}$ and $z''_\nu=1-\frac{1}{z_\nu}$.
Each edge of an ideal tetrahedron has the complex parameter $z_\nu, z'_\nu$ or $z''_\nu$.
The opposite edge of the tetrahedron has the same parameter.
We put $w_\nu=1-z_\nu$.
Then we have $z'_\nu=\frac{1}{w_\nu}$ and $z''_\nu=-\frac{w_\nu}{z_\nu}$.

Let $e_i$ be the $i$th edge of the ideal triangulation of $N$.
Let $p_{i,\nu}$ (resp. $p'_{i,\nu}, p''_{i,\nu}$) be the number of edges of the $\nu$th ideal tetrahedra attached to $e_i$ whose parameter is $z_\nu$ (resp. $z'_\nu, z''_\nu$).
For an edge $e_i$,
we define 
$g_i\ (1\le i\le n)$
 as follows.
\[
g_i=\prod_{\nu=1}^n (z_\nu)^{p_{i,\nu}}(z'_\nu)^{p'_{i,\nu}}(z''_\nu)^{p''_{i,\nu}}
=\prod_{\nu=1}^n (-1)^{p''_{i,\nu}}(z_\nu)^{p_{i,\nu}-p''_{i,\nu}}(w_\nu)^{p''_{i,\nu}-p'_{i,\nu}}
=\prod_{\nu=1}^n (-1)^{p''_{i,\nu}}(z_\nu)^{r'_{i,\nu}}(w_\nu)^{r''_{i,\nu}},
\]
where $r'_{i,\nu}=p_{i,\nu}-p''_{i,\nu}$ and $r''_{i,\nu}=p''_{i,\nu}-p'_{i,\nu}$.
The equation 
$g_i=1$
 is called 
a {\em gluing equation}.
If $g_i=1$
 for any $i$, then the triangulation gives a representation of $\pi_1(N)$ to $\mathrm{PSL}(2,\mathbb C)$.
As $g_1\cdots g_n=1$, 
we can remove one of the equations.
For example, consider the first $n-1$ equations 
$g_1,\ldots,g_{n-1}$.
Let $r_i=(r'_{i,1},  r''_{i,1}, \ldots, r'_{i,n}, r''_{i,n})$.
The {\em matrix $R$} is defined by
\[
R=
\begin{pmatrix}
r_1\\
\vdots\\
r_{n-1}
\end{pmatrix}
=
\begin{pmatrix}
r'_{1,1} & r''_{1,1} & \cdots & r'_{1,n} & r''_{1,n}\\
\vdots & \vdots & & \vdots & \vdots\\
 r'_{n-1,1} & r''_{n-1,1} & \cdots & r'_{n-1,n} & r''_{n-1,n}\\
\end{pmatrix}
\]

Let $\frak{m}$ and $\frak{l}$ be simple closed curves on a component $T$ of $\partial N$ generating $H_1(T, \mathbb Z)$.
We assume that $\frak{m}$ and $\frak{l}$ do not meet the vertices of the triangulation of $T$.
When  $\frak{m}$ goes through a 2-simplex, it meets two edges of the simplex.
We pick up the complex parameter that corresponds to the vertex joining the
two edges of the simplex.
Let 
 $\mu$
 be the product
of these parameters.
We define 
$\lambda$
 for $\frak{l}$ in the same way.
Then we have
\[
\mu
=\pm \prod_{\nu=1}^n (z_\nu)^{m'_\nu}(w_\nu)^{m''_\nu},\ 
\lambda
=\pm \prod_{\nu=1}^n (z_\nu)^{l'_\nu}(w_\nu)^{l''_\nu},
\]

We also use $\frak{m}$ and $\frak{l}$ to denote the vectors $\frak{m}=(m'_1,m''_1,\ldots,m'_n,m''_n)$ and $\frak{l}=(l'_1,l''_1,\ldots,l'_n,l''_n)$.

Let $\mathcal D$ be the affine algebraic set defined by the set of gluing equations.
If points on $\mathcal D$ approach an ideal point,
then $z_\nu$ goes to $0, 1$ or $\infty$.
Let $I=(i_1,\ldots,i_n)$ be a vector where $i_\nu=0, 1$ or $\infty$.
We define $r(I)_{j,k}$ and the {\em degeneration matrix} $R(I)$ as follows.
\[
r(I)_{j,k}
=
\begin{cases}
r''_{j,k}\quad \text{if}\ i_k=1,\\
r'_{j,k}\quad \text{if}\ i_k=0,\\
-r'_{j,k}-r''_{j,k}\quad \text{if}\ i_k=\infty\\
\end{cases}
\]
\[
R(I)=
\begin{pmatrix}
r(I)_{1,1} & \cdots & r(I)_{1,k} & \cdots & r(I)_{1,n}\\
\vdots & & \vdots & & \vdots\\
r(I)_{n-1,1} & \cdots & r(I)_{n-1,k} & \cdots & r(I)_{n-1,n}
\end{pmatrix}
\]

The square matrix $R(I)_k$ is obtained by removing the $k$th column of $R(I)$.
\[
R(I)_k=
\begin{pmatrix}
r(I)_{1,1} & \cdots & \widehat{r(I)_{1,k}} & \cdots & r(I)_{1,n}\\
\vdots & & \vdots & & \vdots\\
r(I)_{n-1,1} & \cdots & \widehat{r(I)_{n-1,k}} & \cdots & r(I)_{n-1,n}
\end{pmatrix}
\]
Here the hat means to delete the entry.
Then we define the {\em degeneration vector} of $I$ by using the determinant of $R(I)_k$'s.
\[
d(I)=(d_1,\ldots,d_n)=
(\det R(I)_1,-\det R(I)_2,\ldots,(-1)^{n-1}\det R(I)_{n})
\]

Let $d'_i=d_i/c$, where $c=\gcd(d_1,\ldots,d_n)$.
In \cite[Theorem 4.1]{Kabaya-JKTR}, it is shown that if all the entries of $d(I)$ are positive (or negative), then there is a corresponding ideal point of $\mathcal D$ and the number of such ideal points is $c$.

Let $v$ be the valuation associated with an ideal point of $\mathcal D$ with the degeneration vector $d=(d_1,\ldots,d_n)$.
Let $\rho_1=(1,0), \rho_0=(0,-1)$ and $\rho_\infty=(-1,1)$.
At an ideal point $(-\log|w_\nu|,\log|z_\nu|)$ diverges to one of these directions.
The valuation of 
$\mu$ and $\lambda$
 can be calculated as follows (Remark 4.4 in \cite{Kabaya-JKTR}).
\[
v(\mu)
=\sum_{\nu=1}^n (m'_\nu v(z_\nu)+m''_\nu v(w_\nu))
=m\wedge (|d'_1|\rho_{i_1}, \ldots,|d'_n|\rho_{i_n}),
\]
\[
v(\lambda)
=\sum_{\nu=1}^n (l'_\nu v(z_\nu)+l''_\nu v(w_\nu))
=l\wedge (|d'_1|\rho_{i_1}, \ldots,|d'_n|\rho_{i_n}),
\]
where
$m=(m'_1,m''_1,\ldots,m'_n,m''_n)$,
$l=(l'_1,l''_1,\ldots,l'_n,l''_n)$, and
 $x\wedge y=\sum_{k=1}^n x'_k y''_k-x''_k y'_k$ for
$x=(x'_1,x''_1,\ldots,x'_n,x''_n)$ and $y=(y'_1,y''_1,\ldots,y'_n,y''_n)$.
If 
$v(\mu)$
 or 
 $v(\lambda)$
 is not zero, then 
 $-\frac{v(\lambda)}{v(\mu)}$ 
 is the boundary slope of an essential surface associated with the ideal point (Theorem 4.6 in \cite{Kabaya-JKTR}).

\section{Proof of Theorem~\ref{thm:Kn}}

In this section, we prove Theorem~\ref{thm:Kn}. We also show that $0$ and $6$ are 
boundary slopes of $K_n$.

We begin by describing a triangulation of the complement of $K_5$ and use 
the method of Kabaya~\cite{Kabaya-JKTR} summarized in Section~\ref{sec:Kab} above
to calculate the boundary slopes.
Next, we deploy the layered solid torus construction due to Jaco and Rubinstein~\cite{JR},
and further developed in the papers of Howie et al.~\cite{HMP, HMPT}. Given a triangulation
of $K_n$, this construction supplies one for $K_{n+2}$. Using induction, we determine 
the triangulation for each $K_n$ ($n \geq 5$, odd).

Kabaya~\cite{Kabaya} himself 
showed that $0$, $6$, and $\pm \frac{10}{3}$ 
are slopes for $K_3$ in 2008. 
Our induction begins with $K_5=K12a815$, which we consider in 
some detail.

\subsection{Triangulation of $K_5$}

In this subsection we provide a triangulation for the complement of $K_5$ and apply Kabaya's method to 
deduce its boundary slopes. 

\begin{figure}
\includegraphics[width=0.44\textwidth]{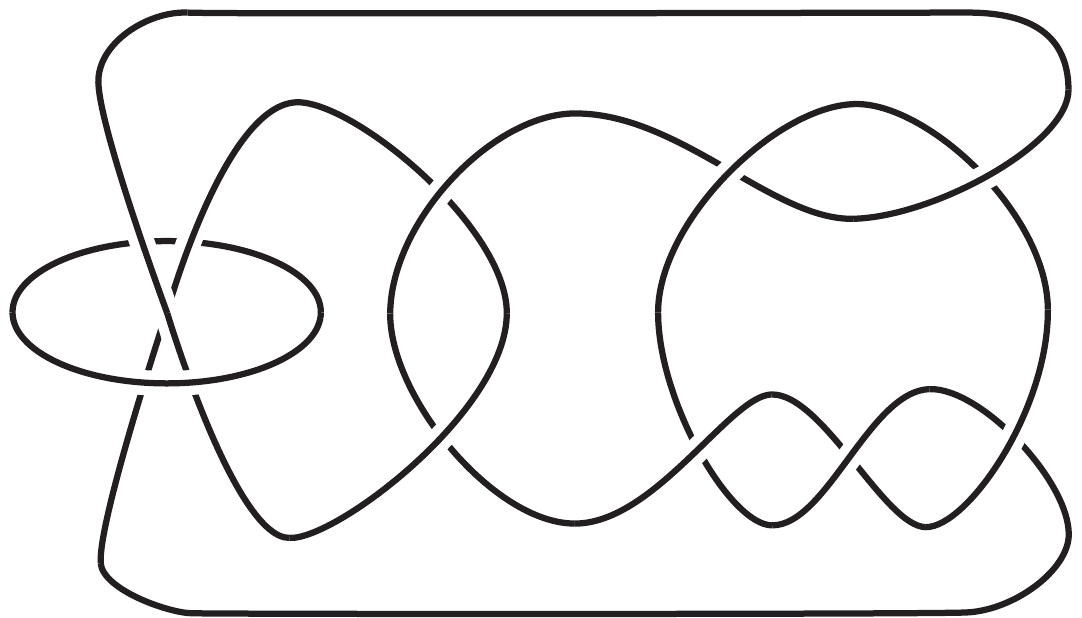}
\caption{
\label{fig:linkKn}
The link $L$. 
The alternating knot $K_n$ for $n \geq 1$ odd is obtained by 
$-\frac{1}{\frac{n+1}{2}}$
- Dehn filling of the trivial component of $L$.}
\end{figure}

\begin{table}
\centering
\begin{tabular}{c|c|c|c|c}
Tetrahedron & Face 012 & Face 013 & Face 023 & Face 123 \\ \hline 
0 & 7(320) & 12(132) & 7(102) & 3(302) \\
1 & 10(320) & 9(302) & 2(102) & 2(132) \\
2 & 1(203) & 6(120) & 4(203) & 1(132) \\
3 & 6(312) & 5(213) & 0(231) & 13(120) \\
4 & 8(102) & 5(012) & 2(203) & 14(132) \\
5 & 4(013) & 7(132) & 8(123) & 3(103) \\
6 & 2(301) & 15(213) & 8(032) & 3(120) \\
7 & 0(203) & 11(310) & 0(210) & 5(031) \\
8 & 4(102) & 10(012) & 6(032) & 5(023) \\
9 & 12(013) & 11(213) & 1(130) & 13(103) \\
10 & 8(013) & 11(012) & 1(210) & 12(120) \\
11 & 10(013) & 7(310) & 13(123) & 9(103) \\
12 & 10(312) & 9(012) & 13(320) & 0(031) \\
13 & 3(312) & 9(213) & 12(320) & 11(023) \\
14 & 15(013) & 15(023) & 15(021) & 4(132) \\
15 & 14(032) & 14(012) & 14(013) & 6(103)
\end{tabular}
\caption{A triangulation of $L$'s complement.}
\label{tab:LTri} 
\end{table}

\begin{figure}
\includegraphics[width=0.9\textwidth]{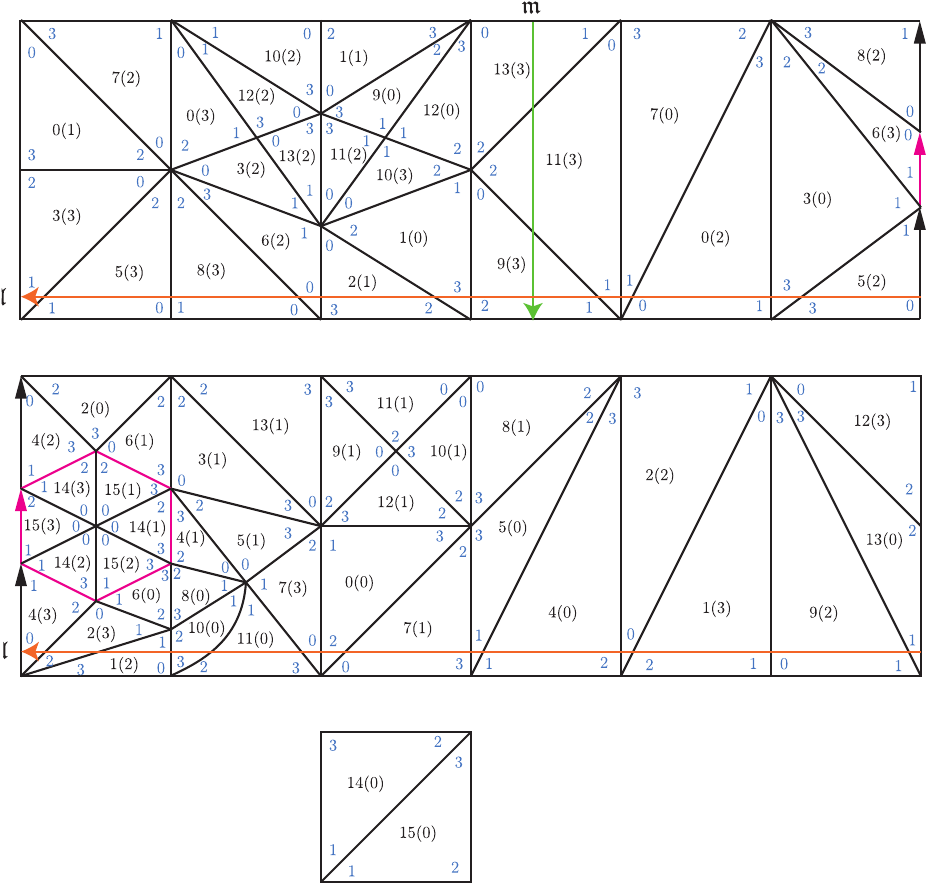}
\caption{A triangulation of the two cusp link complement with homology generators $\frak{l}$
and $\frak{m}$.
The right-hand side of the top figure is identified with the left-hand side of the middle figure.
The bottom is the triangulation of the cusp of the trivial component.}
\label{fig:LTri}
\end{figure}

The knot $K_5$ is the starting point of an induction based on fillings of 
the link $L$ of Figure~\ref{fig:linkKn}. Using SnapPy~\cite{SnapPy} we verify that Table~\ref{tab:LTri} gives an ideal triangulation 
(in Regina~\cite{Regina} notation) for the
complement of the link $L$.
Figure~\ref{fig:LTri}
shows a triangulation of $L$'s complement with generators for the 1-dimensional homology of one boundary torus.

Due to its connection with $L$, the triangulation for $K_5$'s complement is the same except for 
tetrahedra 4, 6, and 14.
(Also, we must remove tetrahedron 15 as there are only 15 tetrahedra for $K_5$, with indices
$0,1,2, \ldots, 14$.)
See Table~\ref{tab:K5tri}.

\begin{table}
\centering
\begin{tabular}{c|c|c|c|c}
Tetrahedron & Face 012 & Face 013 & Face 023 & Face 123 \\ \hline 
4 & 8(102) & 5(012) & 2(203) & 14(103) \\
6 & 2(301) & 14(213) & 8(032) & 3(120) \\
14 & 14(320) & 4(213) & 14(210) & 6(103) 
\end{tabular}
\caption{Identifications for three tetrahedra in $K_5$'s complement}
\label{tab:K5tri} 
\end{table}

The connection is via the layer solid construction that we use in the next subsection (see \cite{HMP,HMPT}).
For now notice that tetrahedra 14 and 15 are part of the second cusp of $L$'s complement and must be modified for 
this reason. In the triangulation for $L$'s complement, 
tetrahedra 14 and 15 are adjacent to 4 and 6, and this is why they are also modified.

\begin{table}
\centering
\begin{tabular}{c|c|l}
Edge & Degree & Tetrahedra (vertices) \\ \hline
0 & 5  & 3 (01), 5 (12), 4 (13), 14 (13), 6 (13) \\
1 & 4 & 0 (13), 12 (23), 13 (02), 3 (23) \\
2 & 8 & 0 (12), 3 (03), 5 (23), 8 (23), 6 (23), 3 (02), 0 (23), 7 (02) \\
3 & 6 & 1 (01), 9 (03), 11 (23), 13 (23), 12 (02), 10(23) \\
4 & 5 & 1 (12), 2 (13), 6 (02), 8 (03), 10 (02) \\
5 & 7 & 1 (03), 2 (12), 1 (13), 9 (02), 12 (03), 13 (03), 9 (23) \\
6 & 4 & 9 (01), 11 (12), 10 (13), 12 (01) \\
7 & 7 & 0 (01), 12 (13), 9 (12), 13 (01), 3 (13), 5 (13), 7 (23) \\
8 & 7 & 1 (02), 10 (03), 11 (02), 13 (12), 3 (12), 6 (12), 2 (01) \\
9 & 7 & 0 (02), 7 (03), 11 (03), 13 (13), 9 (13), 11 (13), 7 (01) \\
10 & 6 & 2 (03), 4 (23), 14 (03), 14 (02), 14 (12), 6 (01) \\
11 & 6 & 4 (01), 5 (01), 7 (13), 11 (01), 10 (01), 8 (01) \\
12 & 6 & 0 (03), 7 (12), 5 (03), 8 (13), 10 (12), 12 (12) \\
13 & 7 & 1 (23), 2 (23), 4 (03), 5 (02), 8 (12), 4 (02), 2 (02) \\
14 & 5 & 4 (12), 14 (01), 14 (23), 6 (03), 8 (02) 
\end{tabular}
\caption{Edge equations for $K_5$}
\label{tab:K5IM} 
\end{table}

From the triangulation of $K_5$'s complement, we deduce the edge equations, which we present in 
Table~\ref{tab:K5IM} using the notation of Regina~\cite{Regina}. 
For example, the 
first row is indicating that the edges of five tetrahedra coincide:
edge $(01)$ of tetrahedron $3$, edge $(12)$ of tetrahedron $5$, 
edge $(13)$ of tetrahedron $4$, edge $(13)$ of tetrahedron $14$,
and edge $(13)$ of tetrahedron $6$.
In the notation of \cite{Kabaya-JKTR} this gives the equation
$$ z_3 w_4^{-1} z_5^{-1} w_5 w_6^{-1} w_{14}^{-1},$$
which becomes the first row of the matrix $R$. 
Note that for the $i$th tetrahedron the edges $(01)$ and $(23)$ correspond to $z$,
the edges $(02)$ and $(13)$ correspond to $w^{-1}$, 
and the edges $(03)$ and $(12)$ correspond to $z^{-1}w$.
To calculate the matrix $R$, shown in Figure~\ref{fig:RK5}, we have omitted the redundant equation for edge 10.

\begin{figure}
\includegraphics[scale=.65]{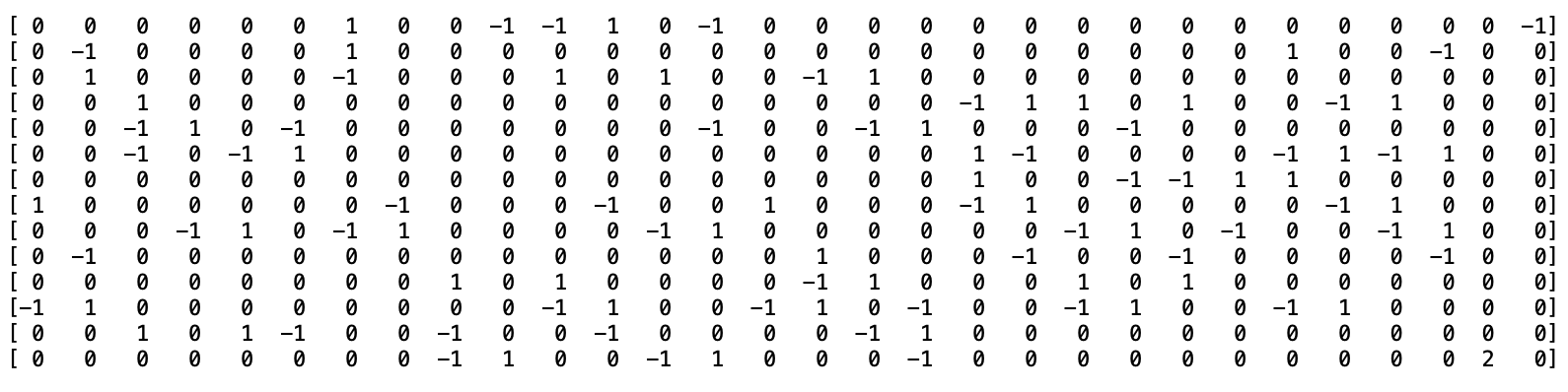}
\caption{The matrix $R$ for $K_5$.}
\label{fig:RK5}
\end{figure}

We can use SnapPy~\cite{SnapPy} to determine
the corresponding vectors for the meridian and longitude
$$\frak{m} = (0, 0, 0, 0, 0, 0, 0, 0, 0, 0, 0, 0, 0, 0, 0, 0, 0, 0, 0, -1, 0, 0, -1, 0, 0, 0, 0, -1, 0, 0)$$
and 
$$\frak{l}' = (-1, -1, 1, -1, 0, -1, -1, 0, 0, 0, 1, -1, 0, -1, 1, 0, -1, 0, 0, 3, -1, 1, 3, -1, 0, 0, 1, 4, 0, 0),$$
where we have replaced the $\frak{l}$ of 
Figure~\ref{fig:LTri} with $\frak{l}' = \frak{l}\frak{m}^{-4}$
to get a homologically trivial longitude.

By considering all $3^{15}$ possibilities, 
we found three lists of degeneration indices that result in 
a degeneration vector $d(I_i)$ whose entries are either 
all positive or all negative (and which, therefore, corresponds to 
a boundary slope).
\begin{align*}
I_1 & = (\infty, 0, 0, 0, \infty, 0, 1, 1, 0, 0, \infty, 0, \infty, \infty, \infty), \\
I_2 & = (\infty, \infty, 0, 0, \infty, 0, 1, \infty, \infty, 1, \infty, 1, \infty, \infty, \infty), \mbox{ and } \\
I_3 & = (\infty, \infty, 0, 0, \infty, 0, 1, \infty, \infty, \infty, \infty, 1, \infty, 0, \infty).
\end{align*}

We will see that the degeneration matrix for $I_1$ is
$$
R(I_1) = \left(
\begin{array}{ccccccccccccccc}
0 & 0 & 0 & 1 & 1 & -1 & -1 & 0 & 0 & 0 & 0 & 0 & 0 & 0 & 1 \\
1 & 0 & 0 & 1 & 0 & 0 & 0 & 0 & 0 & 0 & 0 & 0 & -1 & 1 & 0 \\
-1 & 0 & 0 & -1 & 0 & 1 & 0 & -1 & 1 & 0 & 0 & 0 & 0 & 0 & 0 \\
0 & 1 & 0 & 0 & 0 & 0 & 0 & 0 & 0 & -1 & -1 & 1 & 1 & -1 & 0 \\
0 & -1 & 0 & 0 & 0 & 0 & -1 & 0 & -1 & 0 & 1 & 0 & 0 & 0 & 0 \\
0 & -1 & -1 & 0 & 0 & 0 & 0 & 0 & 0 & 1 & 0 & 0 & 0 & 0 & 0 \\
0 & 0 & 0 & 0 & 0 & 0 & 0 & 0 & 0 & 1 & 1 & -1 & -1 & 0 & 0 \\
-1 & 0 & 0 & 0 & 0 & 0 & 0 & 0 & 0 & -1 & 0 & 0 & 1 & -1 & 0 \\
0 & 0 & 1 & -1 & 0 & 0 & 1 & 0 & 0 & 0 & 0 & 0 & 0 & 0 & 0 \\
1 & 0 & 0 & 0 & 0 & 0 & 0 & 1 & 0 & 0 & 0 & -1 & 0 & 1 & 0 \\
0 & 0 & 0 & 0 & -1 & 1 & 0 & -1 & 1 & 0 & -1 & 1 & 0 & 0 & 0 \\
0 & 0 & 0 & 0 & 0 & -1 & 0 & 1 & 0 & 0 & 0 & 0 & 0 & 0 & 0 \\
0 & 1 & 1 & 0 & 1 & 0 & 0 & 0 & -1 & 0 & 0 & 0 & 0 & 0 & 0 \\
0 & 0 & 0 & 0 & 0 & 0 & 1 & 0 & 0 & 0 & 0 & 0 & 0 & 0 & -2
\end{array}
\right). $$

Indeed, the $0$ entries at positions
2, 3, 4, 6, 9, 10, and 12  in $I_1$ indicate that 
for $i \in \{2, 3, 4, 6, 9, 10, 12\}$ the $i$th column of $R(I_1)$ is the $2i-1$ column 
of the matrix $R$, the column that corresponds to the exponent of $z_i$ in the
gluing equations. 
Similarly, the $1$'s at positions 7 and 8 mean that the $i$th 
column ($i = 7,8$) of $R(I_1)$ is the $2i$ column (the $w_i$ exponents) of the matrix $R$. 
For the remaining $i \in \{1, 5, 11, 13, 14, 15\}$ with entry $\infty$, the $i$th column 
of $R(I_1)$ is the negative of the sum of the $2i-1$ and $2i$ columns of $R$. 

This results in a degeneration vector
$$d(I_1) = (1, 4, 1, 5, 1, 4, 4, 4, 6, 5, 14, 9, 10, 4, 2),$$
which is an alternating sign listing of the minors of $R(I_1)$ formed by deleting the columns
one by one from left to right. Since all entries are positive, this corresponds to a boundary slope.

As described in Section~\ref{sec:Kab} above,
this gives valuations 
$(v(\lambda), v(\mu))
 = (-2,5)$ 
for the meridian and longitude which shows that $\frac{2}{5}$ is a boundary slope of $K_5$.
For $I_2$ we find the degeneration vector is $$(1, 2, 1, 1, 1, 3, 1, 2, 3, 1, 1, 3, 4, 2, 2)$$ with valuations 
$(v(\lambda), v(\mu))
 = (6,-1)$ 
and 
a boundary slope of 6. And $I_3$ gives vector
$$(-1, -2, -1, -1, -1, -3, -1, -2, -3, -1, -1, -2, -2, -1, -2),$$ and 
$(v(\lambda), v(\mu))
= (0,1)$ 
for a boundary slope of $0$.
In summary, we have verified that $0$, $6$, and $\frac{2}{5}$ are boundary slopes of $K_5$.

\begin{figure}
\includegraphics[width=0.9\textwidth]{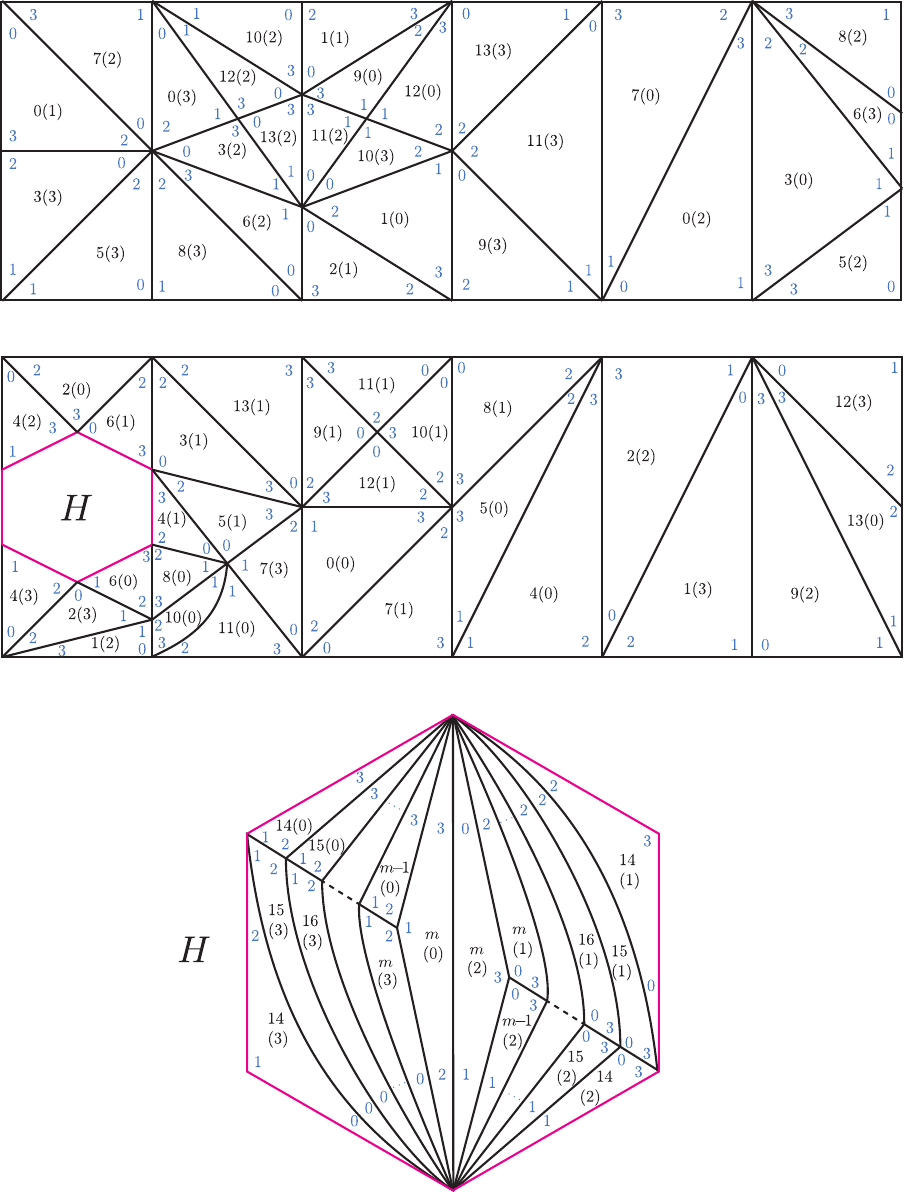}
\caption{Filling the cusp with a layered solid torus gives a triangulation of $K_{n+2}$. 
Here $m=(n+25)/2$.}
\label{fig:Knfill}
\end{figure}

\begin{table}
\centering
\begin{tabular}{c|c|c|c|c}
Tetrahedron & Face 012 & Face 013 & Face 023 & Face 123 \\ \hline 
0 & 7(320) & 12(132) & 7(102) & 3(302) \\
1 & 10(320) & 9(302) & 2(102) & 2(132) \\
2 & 1(203) & 6(120) & 4(203) & 1(132) \\
3 & 6(312) & 5(213) & 0(231) & 13(120) \\
4 & 8(102) & 5(012) & 2(203) & 14(103) \\
5 & 4(013) & 7(132) & 8(123) & 3(103) \\
6 & 2(301) & 14(213) & 8(032) & 3(120) \\
7 & 0(203) & 11(310) & 0(210) & 5(031) \\
8 & 4(102) & 10(012) & 6(032) & 5(023) \\
9 & 12(013) & 11(213) & 1(130) & 13(103) \\
10 & 8(013) & 11(012) & 1(210) & 12(120) \\
11 & 10(013) & 7(310) & 13(123) & 9(103) \\
12 & 10(312) & 9(012) & 13(320) & 0(031) \\
13 & 3(312) & 9(213) & 12(320) & 11(023) \\
14 & 15(312) & 4(213) & 15(013) & 6(103) \\
15 & 16(312) & 14(023) & 16(013) & 14(120) \\
16 & 17(312) & 15(023) & 17(013) & 15(120) \\
17 & 17(320) & 16(023) & 17(210) & 16(120) \\
\end{tabular}
\caption{A triangulation of $K_{11}$'s complement.}
\label{tab:RegK11} 
\end{table}

\begin{table}
\centering
\begin{tabular}{c|c|c|c|c}
Tetrahedron & Face 012 & Face 013 & Face 023 & Face 123 \\ \hline 
17 & 18(312) & 16(023) & 18(013) & 16(120) \\
18 & 18(320) & 17(023) & 18(210) & 17(120) \\
\end{tabular}
\caption{Identifications for two tetrahedra in $K_{13}$'s complement.}
\label{tab:RegK13} 
\end{table}

\subsection{Triangulation of $K_n$}
In this subsection, we use the layered solid torus construction to move from a triangulation
of $K_n$'s complement to one for $K_{n+2}$.
We closely follow the papers of Howie et al.~\cite{HMP, HMPT} to which we refer
the reader for additional details.

Filling the cusp of $L$ with a solid layered torus as in Figure~\ref{fig:Knfill} gives a triangulation
of the knot complement $K_n$. In particular, we see that the complement of $K_{n+2}$ is formed from
that of $K_n$ by adding a single tetrahedron. Tables~\ref{tab:RegK11} and \ref{tab:RegK13} 
illustrate how the triangulation changes with this addition: the triangulations for $K_{11}$ and $K_{13}$
are the same except for the last  two rows of the tables. 
Let's see how adding a tetrahedron as in Tables~\ref{tab:RegK11} and \ref{tab:RegK13} will change the 
boundary slope calculations. 

Given a triangulation of $K_n$ (for $n \geq 5$) we will add a tetrahedron to obtain a triangulation of $K_{n+2}$. As in Tables~\ref{tab:RegK11}
and \ref{tab:RegK13}, let $m = (n+25)/2$ denote the added tetrahedron, glued so that its 012 face matches its 320 face, its 013 face matches
the 023 face of tetrahedra $m-1$, and its 123 face is attached to the 120 face of $m-1$. This change means we will need to modify two edge equations
and add an entirely new one. 

\begin{table}
\centering
\begin{tabular}{c|c|l}
Edge & Degree & Tetrahedra (vertices) \\ \hline
0 & 5  & 3 (01), 5 (12), 4 (13), 14 (13), 6 (13) \\
1 & 4 & 0 (13), 12 (23), 13 (02), 3 (23) \\
2 & 8 & 0 (12), 3 (03), 5 (23), 8 (23), 6 (23), 3 (02), 0 (23), 7 (02) \\
3 & 6 & 1 (01), 9 (03), 11 (23), 13 (23), 12 (02), 10(23) \\
4 & 5 & 1 (12), 2 (13), 6 (02), 8 (03), 10 (02) \\
5 & 7 & 1 (03), 2 (12), 1 (13), 9 (02), 12 (03), 13 (03), 9 (23) \\
6 & 4 & 9 (01), 11 (12), 10 (13), 12 (01) \\
7 & 7 & 0 (01), 12 (13), 9 (12), 13 (01), 3 (13), 5 (13), 7 (23) \\
8 & 7 & 1 (02), 10 (03), 11 (02), 13 (12), 3 (12), 6 (12), 2 (01) \\
9 & 7 & 0 (02), 7 (03), 11 (03), 13 (13), 9 (13), 11 (13), 7 (01) \\
10 & $n+3$ & 2 (03), 4 (23), 6(01), 14 (03), 14 (12), 15 (03), 15 (12), 16 (03), 16 (12), \ldots \\
 &  & \ldots, $(m-1)$ (03), $(m-1)$ (12), $m$ (02), $m$ (03), $m$ (12) \\
11 & 6 & 4 (01), 5 (01), 7 (13), 11 (01), 10 (01), 8 (01) \\
12 & 6 & 0 (03), 7 (12), 5 (03), 8 (13), 10 (12), 12 (12) \\
13 & 7 & 1 (23), 2 (23), 4 (03), 5 (02), 8 (12), 4 (02), 2 (02) \\
14 & 6 & 4 (12), 14 (01), 14 (23), 6 (03), 8 (02), 15 (13) \\
15 & 4 & 14 (02), 15 (23), 16 (13), 15 (01) \\
16 & 4 & 15 (02), 16 (23), 17 (13), 16 (01) \\
17 & 4 & 16 (02), 17 (23), 18 (13), 17 (01) \\
\vdots & \vdots & \vdots  \\
 $m-1$ & 4 & $(m-2)$ (02), $(m-1)$ (23), $m$ (13), $(m-1)$ (01) \\
$m$ & 3 & $(m-1)$ (02), $m$ (23), $m$ (01) 
\end{tabular}
\caption{Edge equations for $K_{n+2}$ for $n \geq 5$ and odd. Here $m = (n+25)/2$.}
\label{tab:Knp2IM} 
\end{table}

Table~\ref{tab:Knp2IM} illustrates the result as we now explain. (cf.~Figure~\ref{fig:Knfill}, which illustrates the situation after adding tetrahedron $m$).
The last edge equation for $K_n$ (e.g., edge 14 in Table~\ref{tab:K5IM}) includes edge (23) of tetrahedron $m-1$. Since that edge will be identified with
edge (13) of tetrahedron $m$, we add $m(13)$ to the last edge equation of $K_n$. Edge Equation 10 will show the identification of 
$2 (03)$, $4 (23)$,  and $6 (10)$ with edges in tetrahedra 14, 15, \ldots, $m-1$. (In Figure~\ref{fig:Knfill}, this is the edge of the top
and bottom vertex of hexagon $H$.) In fact, for all these but the last, $m-1$, the two edges
(03) and (12) of each tetrahedron will be identified with the others. For $m-1$, in addition, edge (02) is part of this equation. We remove $(m-1)$ (02)
from this equation and instead add $m$ (02), $m$ (03) and $m$ (12). 
Due to the way we glued in the new tetrahedron, 
$m$ (12) is identified with $(m-1)$ (12) (see bottom vertex of $H$ in Figure~\ref{fig:Knfill}) and $m$ (03) with $(m-1)$ (03) (top vertex of $H$).
Actually, the top and bottom vertices of $H$ are the same edge, so all four of these are identified as part of edge 10 in
Table~\ref{tab:Knp2IM}.  
Since tetrahedra $m$ is ``folded'' with its 012 and 320 face identified, 
edge (02) of $m$ also belongs with these other edges. On the other hand, we have ``unfolded'' $m-1$ to make way for $m$, so $(m-1)$ (02) is
no longer equivalent to those other edges. Instead, $(m-1)$ (02) is identified with $m$ (01) and $m$ (23) and we add a new edge equation to reflect this.

In summary, in passing from $K_n$ to $K_{n+2}$, we modify the last edge equation of $K_n$, Equation 10 of $K_n$, and add a new edge equation. 
Table~\ref{tab:Knp2IM} illustrates the resulting edge equations for $K_{n+2}$ ($n \geq 5$).

Since the meridian and longitude do not pass through the added tetrahedron
we append two zeros to $\frak{m}$ and $\frak{l}'$ for each tetrahedron added. On the other hand, 
as the linking number of the two components of $L$ is 2 (in absolute value),
to maintain a homologically trivial longitude, we subtract four meridia each time we add a tetrahedron. In other words the meridian and longitude for $K_{n+2}$
are
\begin{align}
\label{eqn:lm}%
\frak{m}&= (0, 0, 0, 0, 0, 0, 0, 0, 0, 0, 0, 0, 0, 0, 0, 0, 0, 0, 0, -1, 0, 0, -1, 0, 0, 0, 0, -1, 0, \ldots, 0) \mbox{ and}  \\
\frak{l}'& = (-1, -1, 1, -1, 0, -1, -1, 0, 0, 0, 1, -1, 0, -1, 1, 0, -1, 0, 0, 2n-3, -1, 1, 2n-3, -1, 0, 0, 1, 2n-2, 0, \ldots, 0) \nonumber
 \end{align}
each terminating in a sequence of $n-1$ 0's. 
As for the sets of degeneration indices, we append an $\infty$ for each tetrahedron added. With these considerations, the degeneration vectors 
follow a nice pattern as we now verify.

\begin{prop} 
\label{prop:I1degv}
Let $n \geq 5$ be odd. The $I_1$ degeneration vector for $K_n$ is
$$(-1)^{\frac{n-1}{2}} (1,n-1,1,n,1,n-1,n-1,n-1, n+1, n, 3n-1, 2n-1, 2n, n-1, n-3, \ldots, 4, 2).$$
\end{prop}
 
\begin{proof} [Proof of Theorem~\ref{thm:Kn}]
Kabaya~\cite{Kabaya} proved that $\pm 10/3$ are boundary slopes of $K_3 = 10_{79}$ and
this was also shown by Dunfield and Garoufalidis~\cite{DG}.
Assume $n \geq 5$. Using the degeneration vector of 
Proposition~\ref{prop:I1degv}
and the meridian and longitude of Equation~\ref{eqn:lm},
we deduce the following valuations and boundary slope
by applying Theorems 4.1 and 4.6 in \cite{Kabaya-JKTR}: 
$(v(\lambda), v(\mu))
 = (-1)^{\frac{n-1}{2}}(-(2n^2-10n+2),n)$ with boundary slope $\frac{2n^2-10n+2}{n}$.
\end{proof}

\begin{rem} 
\label{rem:I2degv}
Let $n \geq 5$ be odd. The $I_2$ degeneration vector for $K_n$ is
$$(-1)^{\frac{n-1}{2}} (1,2,1,1,1,n-2,1,2, n-2, 1, 1, 3, 4, 2, n-3, n-5, \ldots, 2),$$
and that for $I_3$ is
$$(-1)^{\frac{n+1}{2}} (1,2,1,1,1,n-2,1,2, n-2, 1, 1, 2, 2, 1, n-3, n-5, \ldots, 2).$$
For $I_2$: 
$(v(\lambda), v(\mu))
 = (-1)^{\frac{n+1}{2}}(-6,1)$ with boundary slope 6; and 
for $I_3$: 
$(v(\lambda), v(\mu))
 = (-1)^{\frac{n-1}{2}}(0,1)$ with boundary slope 0.
\end{rem}

It remains to prove Proposition~\ref{prop:I1degv}. 
We begin with a lemma.

\begin{lem}
\label{lem:RI1}%
Let $n \geq 5$ be odd.
The $R(I_1)$ matrix of $K_{n+2}$ is formed from that for $K_n$ by adding the row
$(0, 0, 0, \ldots 0, 1, -2)$ at bottom and column $(0,0,0, \ldots, 0,1,-2)^T$ at right.
\end{lem}

\begin{proof}
Let $R_n$ denote the matrix $R$ for $K_n$ and let $m = (n+25)/2$. 
By induction, $R_n$ is $(m-1) \times 2m$. Since we have formed $R$ omitting the equation for edge 10, 
we can see the difference between $R_{n+2}$ and $R_n$ by
keeping track of the changes to the edge equations
as in our discussion of Table~\ref{tab:Knp2IM}
above: we modify the last edge equation for $K_n$ and add an entirely new one.
Comparing the two matrices, we can think of $R_{n+2}$ as being formed from $R_n$
by appending two columns at right and a row at bottom that are zero
except for three entries. In $R_{n+2}$, the three non-zero entries occur in the second last row and last column: $(R_{n+2})_{m-1, 2m+2} = -1$,
and in the last row and third and second to last columns: $(R_{n+2})_{m,2m} = -1$ and $(R_{n+2})_{m,2m+1} = 2$.  
The entries in the last row correspond to the added edge equation
with edges $(m-1)$ (02), $m$ (23), $m$ (01) while
the $-1$ in the second to last row corresponds to adding $m$ (13) to the last edge equation for $K_n$.

Since $n \geq 5$, 
the last two terms in the list of degeneration indices $I_1$ are both $\infty$,
which means that the last two columns of the $R(I_1)$ matrix
for $K_{n+2}$ combine the negatives of the last four columns of $R_{n+2}$. Since the last row of
$R_{n+2}$ is $(0,0, \ldots, 0, -1, 2,0)$, the last row of $R(I_1)$ is 
$$(0,0,0, \ldots, 0, (-1)0 + (-1)(-1), (-1)2 + (-1)0) = (0,0,0, \ldots, 0,1,-2)$$
as required. The last column of $R(I_1)$ comes from the sum of the negative of the last two columns of $R_{n+2}$,
which are all zeros, except for the last entry, which becomes $-2$ as we just noticed, and the penultimate
which combines $0$ and $-1$ to give $(-1)0 + (-1)1 = 1$. Thus, the last column of $R(I_1)$ is also 
of the requisite form.
\end{proof}

\begin{proof}[Proof of Proposition~\ref{prop:I1degv}]
We use a two step induction. 

As mentioned above,
$$\mathbf{d}_5 = (1, 4, 1, 5, 1, 4, 4, 4, 6, 5, 14, 9, 10, 4, 2)$$
and, by a similar, direct calculation,
$$\mathbf{d}_7 = 
(-1, -6, -1, -7, -1, -6, -6, -6, -8, -7, -20, -13, -14, -6, -4, -2),
$$
which establishes the base cases for the induction.

Assume for $n \geq 5$, we have degeneration vectors 
$$\mathbf{d}_{n} = (-1)^{\frac{n-1}{2}} (1,n-1,1,n,1,n-1,n-1,n-1, n+1, n, 3n-1, 2n-1, 2n, n-1, n-3 \ldots, 2)$$
and 
$$\mathbf{d}_{n+2}= (-1)^{\frac{n+1}{2}} (1,n+1,1,n+2,1,n+1,n+1,n+1, n+3, n+2, 3n+5, 2n+3, 2n+4, n+1, n-1, n-3, \ldots, 2).$$
By expanding along the last row, $(0, 0, 0, \ldots 0, 1, -2)$ of $R(I_1)$ for $K_{n+4}$ (as in the Lemma) and, when necessary, the 
last column $(0, 0, 0, \ldots 0, 1)^T$ of the resulting submatrix, we find a simple relationship 
between the degeneration vectors.
The degeneration vector for $K_{n+4}$ is
\begin{align*}
\mathbf{d}_{n+4} = & -2(\mathbf{d}_{n+2},0) - (\mathbf{d}_{n},0,0) +  (-1)^{\frac{n+3}{2}} (\mathbf{0},0,2) \\
             = & (-1)^{\frac{n+1}{2}} (-2+1, -2(n+1)+n-1, -2+1, -2(n+2) + n,-2+1,-2(n+1)+n-1, \\
                 & \mbox{ } -2(n+1)+n-1,-2(n+1)+n-1, -2(n+3) + n+1,  -2(n+2) + n,\\ 
                 & \mbox{ } -2(3n+5) + 3n-1, -2(2n+3) + 2n-1, -2(2n+4) + 2n, \\
                 & \mbox{ } -2(n+1) + n-1, -2(n-1) + n-3, -2(n-3) + n-5, \ldots, -4,-2) \\
               = & (-1)^{\frac{n+3}{2}} (1, n+3, 1, n+4,1,n+3,n+3,n+3, n+5,  n+4, 3n+11, \\
                  & \hspace{6 cm} 2n+7, 2n+8, n+3, n+1, n-1, \ldots, 4,2) 
\end{align*}
as required.
\end{proof}
               
\section{Proof of Theorem~\ref{thm:Jn}}

In this section we prove Theorem~\ref{thm:Jn}. 
The method is largely similar to that  for Theorem~\ref{thm:Kn} and we omit some of the details. 

\begin{table}
\centering
\begin{tabular}{c|c|c|c|c}
Tetrahedron & Face 012 & Face 013 & Face 023 & Face 123 \\ \hline 
0 & 7(123) & 10(302) & 15(301) & 8(023) \\
1 & 3(013) & 14(321) & 13(231) & 2(012) \\
2 & 1(123) & 3(012) & 4(123) & 4(120) \\
3 & 2(013) & 1(012) & 15(312) & 14(302) \\
4 & 2(312) & 10(132) & 5(320) & 2(023) \\
5 & 11(312) & 6(012) & 4(320) & 7(120)  \\
6 & 5(013) & 9(201) & 10(021) & 8(103) \\
7 & 5(312) & 8(012) & 11(320) & 0(012) \\
8 & 7(013) & 6(213) & 0(123) & 16(132) \\
9 & 6(130) & 14(103) & 12(123) & 10(103)  \\
10 & 6(032) & 9(213) & 0(130) & 4(031) \\
11 & 12(032) & 12(012) & 7(320) & 5(120) \\
12 & 11(013) & 17(103) & 11(021) & 9(023) \\
13 & 17(203) & 14(012) & 15(032) & 1(302) \\
14 & 13(013) & 9(103) & 3(231) & 1(310) \\
15 & 16(320) & 0(230) & 13(032) & 3(230) \\
16 & 17(102) & 17(132) & 15(210) & 8(132) \\
17 & 16(102) & 12(103) & 13(102) & 16(031)
\end{tabular}
\caption{A triangulation of $J_2$.}
\label{tab:J2} 
\end{table}

\begin{table}
\centering
\begin{tabular}{c|c|c|c|c}
Tetrahedron & Face 012 & Face 013 & Face 023 & Face 123 \\ \hline 
0 & 7(123) & 10(302) & 15(301) & 8(023) \\
\vdots & \vdots & \vdots & \vdots & \vdots \\
5 & 11(312) & 6(012) & 4(320) & 7(120)  \\
6 & 5(013) & 9(201) & 18(021) & 8(103) \\
7 & 5(312) & 8(012) & 11(320) & 0(012) \\
8 & 7(013) & 6(213) & 0(123) & 16(132) \\
9 & 6(130) & 14(103) & 12(123) & 10(103)  \\
10 & 18(013) & 9(213) & 0(130) & 4(031) \\
11 & 12(032) & 12(012) & 7(320) & 5(120) \\
\vdots & \vdots & \vdots & \vdots & \vdots \\
17 & 16(102) & 12(103) & 13(102) & 16(031) \\
18 & 6(032) & 10(012) & 19(021) & 19(103) \\
19 & 18(032) & 18(213) & 20(021) & 20(103) \\
20 & 19(032) & 19(213) & 20(231) & 20(302) 
\end{tabular}
\caption{A triangulation of $J_8$.}
\label{tab:J8} 
\end{table}

\begin{table}
\centering
\begin{tabular}{c|c|c|c|c}
Tetrahedron & Face 012 & Face 013 & Face 023 & Face 123 \\ \hline 
6 & 5(013) & 9(201) & 18(021) & 8(103) \\
10 & 18(013) & 9(213) & 0(130) & 4(031) \\
18 & 6(032) & 10(012) & 18(231) & 18(302) 
\end{tabular}
\caption{A triangulation of $J_4$ showing only the tetrahedra that have changed compared to $J_2$.}
\label{tab:J4} 
\end{table}

\begin{table}
\centering
\begin{tabular}{c|c|c|c|c}
Tetrahedron & Face 012 & Face 013 & Face 023 & Face 123 \\ \hline 
6 & 5(013) & 9(201) & 18(021) & 8(103) \\
10 & 18(013) & 9(213) & 0(130) & 4(031) \\
18 & 6(032) & 10(012) & 19(021) & 19(103) \\
19 & 18(032) & 18(213) & 19(231) & 19(302) 
\end{tabular}
\caption{A triangulation of $J_6$ showing only the tetrahedra that have changed compared to $J_2$.}
\label{tab:J6} 
\end{table}

We begin with an 18 tetrahedron triangulation of the exterior of $J_2$, given in Table~\ref{tab:J2}. We will insert the layered torus along
the hexagon formed where the $023$ face of tetrahedron 6 is identified with the $012$ face of tetrahedron 10. For example, Table~\ref{tab:J8} 
gives the resulting triangulation for $J_8$, where we have omitted many of the rows that show no change. In fact, aside from the added lines for
the new tetrahedra 18, 19, and 20, the only changes are in the rows for tetrahedra 6 and 10. As further examples, we include Table~\ref{tab:J4}
and \ref{tab:J6} which list only the tetrahedra of $J_4$ and $J_6$ that have changed compared to $J_2$.

\begin{table}
\centering
\begin{tabular}{c|c|l}
Edge & Degree & Tetrahedra (vertices) \\ \hline
0 & 7  & 0 (01), 10 (03), 9 (23), 12 (23), 11 (12), 5 (12), 7 (12)  \\
1 & 7 & 0 (02), 7 (13), 8 (12), 16 (13), 17 (23), 13 (02), 15 (03)  \\
2 & 6 & 0 (03), 15 (13), 3 (02), 2 (03), 4 (13), 10 (23) \\
3 & 7 & 0 (12), 8 (02), 7 (03), 11 (03), 12 (02), 11 (02), 7 (23)  \\
4 & 7 & 0 (13), 10 (02), 18 (03), 19 (01), 18 (12), 6 (23), 8 (03) \\
5 & 4 & 0 (23), 8 (23), 16 (32), 15 (01) \\
6 & 6 & 1 (01), 14 (23), 3 (13), 1 (12), 2 (01), 3 (01) \\
7 & 4 & 1 (02),   3 (03), 15 (23), 13 (23) \\
8 & 6 & 1 (03),  13 (12), 17 (03), 12 (13), 9 (03), 14 (13) \\
9 & 7 & 1 (13), 14 (12), 13 (13), 1 (23), 2 (12), 4 (12), 2 (02)  \\
10 & 6 & 2 (13), 3 (12), 14 (03), 9 (13), 10 (13), 4 (01)  \\
11 & 7 & 2 (23), 4 (02), 5 (23), 7 (02), 11 (23), 5 (02), 4 (23) \\
12 & 8 & 3 (23), 14 (02), 13 (03), 15 (02), 16 (03), 17 (12), 16 (02), 15 (12) \\
13 & 11 & 4 (03), 5 (03), 6 (02), 10 (12), 18 (02), 18 (13), 19 (02), 19 (13), 20 (02), 20 (13), 20 (23) \\
14 & 5 & 5 (01), 6 (01), 9 (02), 12 (12), 11 (13) \\
15 & 4 & 5 (13), 6 (12), 8 (01), 7 (01) \\
16 & 4 & 6 (03), 18 (01), 10 (01), 9 (12) \\
17 & 7 & 6 (13), 9 (01), 14 (01), 13 (01), 17 (02), 16 (12), 8 (13) \\
18 & 6 & 11 (01), 12 (01), 17 (01), 16 (01), 17 (13), 12 (03) \\
19 & 4 & 18 (23), 19 (03), 19 (12), 20 (01) \\
20 & 3 & 19 (23), 20 (03), 20 (12) 
\end{tabular}
\caption{Edge equations for $J_8$}
\label{tab:J8IM} 
\end{table}

\begin{table}
\centering
\begin{tabular}{c|c|l}
Edge & Degree & Tetrahedra (vertices) \\ \hline
0 & 7  & 0 (01), 10 (03), 9 (23), 12 (23), 11 (12), 5 (12), 7 (12)  \\
1 & 7 & 0 (02), 7 (13), 8 (12), 16 (13), 17 (23), 13 (02), 15 (03)  \\
2 & 6 & 0 (03), 15 (13), 3 (02), 2 (03), 4 (13), 10 (23) \\
3 & 7 & 0 (12), 8 (02), 7 (03), 11 (03), 12 (02), 11 (02), 7 (23)  \\
4 & 7 & 0 (13), 10 (02), 18 (03), 19 (01), 18 (12), 6 (23), 8 (03) \\
5 & 4 & 0 (23), 8 (23), 16 (32), 15 (01) \\
6 & 6 & 1 (01), 14 (23), 3 (13), 1 (12), 2 (01), 3 (01) \\
7 & 4 & 1 (02),   3 (03), 15 (23), 13 (23) \\
8 & 6 & 1 (03),  13 (12), 17 (03), 12 (13), 9 (03), 14 (13) \\
9 & 7 & 1 (13), 14 (12), 13 (13), 1 (23), 2 (12), 4 (12), 2 (02)  \\
10 & 6 & 2 (13), 3 (12), 14 (03), 9 (13), 10 (13), 4 (01)  \\
11 & 7 & 2 (23), 4 (02), 5 (23), 7 (02), 11 (23), 5 (02), 4 (23) \\
12 & 8 & 3 (23), 14 (02), 13 (03), 15 (02), 16 (03), 17 (12), 16 (02), 15 (12) \\
13 & $n+5$ & 4 (03), 5 (03), 6 (02), 10 (12), 18 (02), 18 (13), 19 (02), 19 (13), 20 (02), 20 (13), \ldots \\
 &  & \ldots, $(m-1)$ (02), $(m-1)$ (13), $m$ (02), $m$ (13), $m$ (23) \\
14 & 5 & 5 (01), 6 (01), 9 (02), 12 (12), 11 (13) \\
15 & 4 & 5 (13), 6 (12), 8 (01), 7 (01) \\
16 & 4 & 6 (03), 18 (01), 10 (01), 9 (12) \\
17 & 7 & 6 (13), 9 (01), 14 (01), 13 (01), 17 (02), 16 (12), 8 (13) \\
18 & 6 & 11 (01), 12 (01), 17 (01), 16 (01), 17 (13), 12 (03) \\
19 & 4 & 18 (23), 19 (03), 19 (12), 20 (01) \\
20 & 4 & 19 (23), 20 (03), 20 (12), 21 (01) \\
\vdots & \vdots & \vdots  \\
 $m-1$ & 4 & $(m-2)$ (23), $(m-1)$ (03), $(m-1)(12)$, $m$ (01) \\
$m$ & 3 & $(m-1)$ (23), $m$ (03), $m$ (12) 
\end{tabular}
\caption{Edge equations for $J_{n+2}$ for $n \geq 6$ and even. Here 
$m = (n+34)/2$.}
\label{tab:Jnp2IM} 
\end{table}

Table~\ref{tab:J8IM} gives the edge equations for $J_8$. Notice the similarity between the last two edge equations. Here, Equation 13 
is the analogue of Equation 10 in the previous section. We obtain a triangulation of the exterior of $J_{10}$ by adding tetrahedron 21. This 
entails adding $21(01), 21 (13), 21 (23)$ to the edge 
Equation 13 and then removing $20 (23)$. In addition, we must add $21 (01)$ to Equation 20 and add a new equation, Equation 21, with three edges: 
$20 (23), 21 (03), 21 (12)$. 

Table~\ref{tab:Jnp2IM} shows the general form of 
the edge equations for $J_{n+2}$, after adding tetrahedron $m = (n+34)/2$.
Assuming we construct matrices without using the redundant Equation 13, the change in going from the matrix $R$ for $J_n$ to that for $J_{n+2}$
occurs in the last few rows and columns, corresponding to adding $m(01)$ to Equation $m-1$ and introducing a new Equation $m$ as in the table; 
see Lemma~\ref{lem:RIJ} below.

In contrast to the $K_n$ sequence, the components of the link generating the sequence $J_n$ have linking number 0 and
we don't need to modify the longitude when adding tetrahedra, except for appending 0's. 
Using SnapPy~\cite{SnapPy} in Sage~\cite{Sage} we find that the meridian and longitude for $J_n$ ($n \geq 2$, even) are
\begin{align*}
\frak{m}&= (0, 0, 0, 0, 0, 0, 0, 0, -1, 0, 0, 0, 0, 0, 0, 0, 0, 0, 0, 0, 0, 0, 0, 0, 0, 0, 0, 0, 0, 0, 1, -1, 0, 1, 1, 0, 0, \ldots, 0) \mbox{ and}  \\
\frak{l}& = (1, -1, 1, -2, -1, 1, -1, 1, -8, -1, -1, 2, -2, 1, 1, -1, -2, 1, -1, 1, -2,  \\ 
&  \hspace{2 in} 3, 0, -1, 0, 1, 0, 1, -1, -1, 10, -10, 1, 7, 7, 2, 0, 0, 0, \ldots, 0) 
 \end{align*}
where the meridian (respectively, longitude) terminates in a sequence of $n-1$ (resp.~$n-2$) zeroes.

The two lists of degeneracy indices for $J_2$ corresponding to the boundary slopes $-\frac{14(2) - 2}{2} = -13$ 
and $-\frac{10(2)+8}{2+1} = -28/3$ are, respectively,
\begin{align*}
I_1 & = (\infty, \infty, 0, \infty, \infty, 0, 1, 0, 1, 0, 1, 0, \infty, 0, 1, \infty, 0, 0), \mbox{ and } \\
I_2 & = (\infty, 1, 0, \infty, \infty, 0, 1, 0, 0, 0, 1, 0, \infty, \infty, 1, \infty, 0, 0). \\
\end{align*}
We have chosen $I_1$ and $I_2$ so that the degeneration vectors $d(I_i)$ have all entries negative and, therefore,
each corresponds to a boundary slope.
Thereafter, to realize the boundary slopes of Theorem~\ref{thm:Jn} we append a 0 to the list each time we add a tetrahedron to the triangulation.

\begin{prop} 
\label{prop:JI1degv}
Let $n \geq 6$ be even. The $I_1$ degeneration vector for $J_n$ is
$$(-1)^{\frac{n}{2}} (2n-1, n, n, 1, n-1, n+1, n-1, 2n-1, n, 3n-2, 2n, n-1, n-1, n-1, 1, 1, 1, 1, n-2, n-4, \ldots, 4, 2).$$
\end{prop}
 
\begin{prop} 
\label{prop:JI2degv}
Let $n \geq 6$ be even. The $I_2$ degeneration vector for $J_n$ is
$$(-1)^{\frac{n}{2}} (2n+1,n+1,n+1,1,n,n+2,n,2n,n,2n,3n+3,n,2n+1,1,n+2,1,1,1,n-2,n-4, \ldots, 4, 2).$$
\end{prop}

\begin{proof} [Proof of Theorem~\ref{thm:Jn}]
Assume $n \geq 6$. Using the degeneration vectors of the last 
two Propositions 
and the meridian and longitude given above,
we deduce the following valuations and boundary slopes. 
For $I_1$:
$(v(\lambda), v(\mu)) = (-1)^{\frac{n+2}{2}}(14n-2,n)$ 
 with boundary slope $ - \frac{14n-2}{n}$; and 
for $I_2$: 
$(v(\lambda), v(\mu)) = (-1)^{\frac{n+2}{2}}(10n+8,n+1)$ 
with boundary slope $- \frac{10n+8}{n+1}$.

For $n = 2$, using the list of degeneration indices $I_1$ given above, by direct calculation, the degeneration vector is 
$d(I_1) =   -(3, 2, 2, 1, 1, 3, 1, 3, 2, 4, 4, 1, 1, 1, 1, 1, 1, 1)$ 
giving valuation $(v(\lambda), v(\mu)) = (26,2)$ 
and boundary slope $-13$.
Using $I_2$, we have 
$d(I_2) = -(5, 3, 3, 1, 2, 4, 2, 4, 2, 4, 9, 2, 5, 1, 4, 1, 1, 1)$  
with valuation $(v(\lambda), v(\mu)) = (28,3)$
and boundary slope $-28/3$. 
For $n = 4$, we have, with indices $I_1$: degeneration vector $(7, 4, 4, 1, 3, 5, 3, 7, 4, 10, 8, 3, 3, 3, 1, 1, 1, 1, 2)$,
valuation $(-54,-4)$, and slope $-27/2$; and, with indices $I_2$:  degeneration vector $(9, 5, 5, 1, 4, 6, 4, 8, 4, 8, 15, 4, 9, 1, 6, 1, 1, 1, 2)$,
valuation $(-48,-5)$, and slope $-48/5$. 
\end{proof}

It remains to prove Propositions~\ref{prop:JI1degv} and \ref{prop:JI2degv}.

\begin{lem}
\label{lem:RIJ}%
Let $n \geq 6$ be even.
The $R(I_1)$ matrix of $J_{n+2}$ is formed from that for $J_n$ by adding the row
$(0, 0, 0, \ldots 0,1,-2)$ at bottom and column $(0,0,0, \ldots, 0,1,-2)^T$ at right.
Similarly, the $R(I_2)$ matrix of $J_{n+2}$ comes from that for $J_n$ by adding the row
$(0, 0, 0, \ldots 0,1,-2)$ at bottom and column $(0,0,0, \ldots, 0,1,-2)^T$ at right.
\end{lem}

\begin{proof}
Let $R_n$ denote the matrix $R$ for $J_n$ and let $m = (n+34)/2$. 
By induction, $R_n$ is $(m-1) \times 2m$. Since we have formed $R$ omitting edge Equation 13, 
constructing $R_{n+2}$ from $R_n$ only involves changes in the last two edge equations: 
we add $m (01)$ to the last edge equation for $J_n$ and create an entirely new equation
with three terms: $(m-1) (23)$, $m (03)$, and $m (12)$,
This means we can build $R_{n+2}$ from $R_n$
by appending two columns at right and a row at bottom that are zero
except for four entries. 
In $R_{n+2}$, the four non-zero entries are $(R_{n+2})_{m-1, 2m+1} = 1$, $(R_{n+2})_{m,2m-1} = 1$
and $(R_{n+2})_{m,2m+1} = -2$ and $(R_{n+2})_{m,2m+2} = 2$. 

Since $n \geq 6$, 
the last two terms in the list of degeneration indices $I_1$ are both $0$,
which means that the last two columns of the $R(I_1)$ matrix
for $K_{n+2}$ are columns $2m-1$ and $2m+1$ of $R_{n+2}$. Since the last row of
$R_{n+2}$ is $(0,0, \ldots, 0, 1, 0, -2, 2)$, the last row of $R(I_1)$ is $$(0,0,0, \ldots, 0,1,-2)$$
as required. The last column of $R(I_1)$ is the second to last column of $R_{n+2}$,
which is all zeros, except for the last entry, $-2$, and the penultimate,
$1$. Thus, the last column of $R(I_1)$ is also of the requisite form.

The same argument shows that the last row and column of $R(I_2)$ are again
$(0, 0, 0, \ldots 0,1,-2)$ and $(0,0,0, \ldots, 0,1,-2)^T$, respectively.
\end{proof}

\begin{proof} [Proof of Proposition~\ref{prop:JI1degv}]
We use a two step induction. 

By direct calculation,
$$\mathbf{d}_6 = (-11, -6, -6, -1, -5, -7, -5, -11, -6, -16, -12, -5, -5, -5, -1, -1, -1, -1, -4, -2)$$
and
$$\mathbf{d}_8 = 
(15, 8, 8, 1, 7, 9, 7, 15, 8, 22, 16, 7, 7, 7, 1, 1, 1, 1, 6, 4, 2),
$$
which establishes the base cases for the induction.

Assume for $n \geq 6$, we have degeneration vectors 
$$\mathbf{d}_n = (-1)^{\frac{n}{2}} (2n-1, n, n, 1, n-1, n+1, n-1, 2n-1, n, 3n-2, 2n, n-1, n-1, n-1, 1, 1, 1, 1, n-2, n-4, \ldots, 4, 2)$$
and 
\begin{align*}
\mathbf{d}_{n+2} =& (-1)^{\frac{n+2}{2}} (2n+3, n+2, n+2, 1, n+1, n+3, n+1, 2n+3, n+2, 3n+4, 2n+4, \\
 & \hspace{7 cm} n+1, n+1, n+1, 1, 1, 1, 1, n, n-2, \ldots, 4, 2).
 \end{align*}
The degeneration vector for $J_{n+4}$ is
\begin{align*}
\mathbf{d}_{n+4} = & -2(\mathbf{d}_{n+2},0) - (\mathbf{d}_{n},0,0) + (-1)^{\frac{n+4}{2}} (\mathbf{0}, 0, 2) \\
             = & (-1)^{\frac{n+4}{2}} (2(2n+3)-(2n-1), 2(n+2)-n, 2(n+2) - n, 2(1) -1, 2(n+1)-(n-1),  \\
                 & \mbox{ }  2(n+3)-(n+1), 2(n+1)-(n-1), 2(2n+3)- (2n-1), 2(n+2) -n,   2(3n+4) - (3n-2),\\ 
                 & \mbox{ } 2(2n+4) - 2n, 2(n+1) - (n-1), 2(n+1) - (n-1),  2(n+1) - (n-1),\\
                 & \mbox{ } 2(1) -1, 2(1) - 1, 2(1) - 1, 2(1) - 1, 2n-(n-2), 2(n-2) -  (n-4), \ldots, 4,2) \\
               = & (-1)^{\frac{n+4}{2}} (2n+7, n+4, n+4, 1,n+3, n+5, n+3, 2n+7, n+4,  3n+10, 2n+8, \\
                  & \hspace{7 cm} n+3, n+3, n+3, 1,1,1,1, n+2, n, \ldots, 4,2) 
\end{align*}
as required.
\end{proof}

\begin{proof} [Proof of Proposition~\ref{prop:JI2degv}]
By direct calculation, we confirm the base cases
$$\mathbf{d}_6 = (-13, -7, -7, -1, -6, -8, -6, -12, -6, -12, -21, -6, -13, -1, -8, -1, -1, -1, -4, -2)$$
and
$$\mathbf{d}_8 = 
(17, 9, 9, 1, 8, 10, 8, 16, 8, 16, 27, 8, 17, 1, 10, 1, 1, 1, 6, 4, 2).
$$

Assume for $n \geq 6$, we have degeneration vectors 
$$\mathbf{d}_n =  (-1)^{\frac{n}{2}} (2n+1,n+1,n+1,1,n,n+2, n, 2n, n, 2n, 3n+3, n, 2n+1,1,n+2,1,1,1,n-2,n-4, \ldots, 4, 2).$$
and 
\begin{align*}
\mathbf{d}_{n+2} =& (-1)^{\frac{n+2}{2}} (2n+5, n+3, n+3, 1, n+2, n+4, n+2, 2n+4, n+2, 2n+4, \\
 & \hspace{6.5 cm} 3n+9, n+2, 2n+5, 1, n+4, 1, 1, 1, n, n-2, \ldots, 4, 2).
 \end{align*}
The degeneration vector for $J_{n+4}$ is
\begin{align*}
\mathbf{d}_{n+4} = & -2(\mathbf{d}_{n+2},0)  - (\mathbf{d}_{n},0,0) +(-1)^{\frac{n+4}{2}}  (\mathbf{0},0,2) \\
             = & (-1)^{\frac{n+4}{2}} (2(2n+5)-(2n+1), 2(n+3)-(n+1), 2(n+3) - (n+1), 2(1) -1, 2(n+2)-n,  \\
                 & \mbox{ }  2(n+4)-(n+2), 2(n+2)-n, 2(2n+4)- 2n, 2(n+2) -n,   2(2n+4) - 2n,\\ 
                 & \mbox{ } 2(3n+9) - (3n+3),  2(n+2) - n,  2(2n+5) -(2n+1), 2(1) - 1, 2(n+4) - (n+2),  \\
                 & \mbox{ } 2(1) - 1, 2(1) - 1, 2(1) - 1, 2n-(n-2), 2(n-2) -  (n-4), \ldots, 4,2) \\
               = & (-1)^{\frac{n+4}{2}} (2n+9, n+5, n+5, 1,n+4, n+6, n+4, 2n+8, n+4,  2n+8, 3n+15, \\
                  & \hspace{7 cm} n+4, 2n+9, 1, n+6, 1,1,1, n+2, n, \ldots, 4,2) 
\end{align*}
as required.
\end{proof}


\begin{thebibliography}{10}

\bibitem{Regina}
B.A.\ Burton, R.\ Budney, W.\ Pettersson, et al.,
{\em Regina: Software for low-dimensional topology},
{\tt http://regina-normal.github.io/}, 1999-2021. 

\bibitem{SnapPy}
M.\ Culler, N.M.\ Dunfield, M.\ Goerner, and J.R.\ Weeks, SnapPy, a computer program for studying the geometry and topology of 3-manifolds, 
{\tt http://snappy.computop.org} (07/29/2022).

\bibitem{DG}
N.\ Dunfield and S.\ Garoufalidis,
Incompressibity criteria for spun-normal surfaces.
Trans. Amer. Math. Soc. \textbf{364} (2012), 6109-6137.
arXiv:1102.4588

\bibitem{AGT}
D.\ Futer, M.\ Ishikawa, Y.\ Kabaya, T.W.\ Mattman, and K.\ Shimokawa,
Finite surgeries on three-tangle pretzel knots,
Algebr. Geom. Topol. \textbf{9} (2009), 743--771.

\bibitem{HT}
A.\ Hatcher and W.\ Thurston,
Incompressible surfaces in 2-bridge knot complements.
Invent.\ Math.\  \textbf{79}  (1985), 225--246.

\bibitem{HO}
A. Hatcher and U. Oertel,
Boundary slopes for Montesinos knots.
Topology  \textbf{28}  (1989),  453--480.

\bibitem{HMP} J.\ Howie, D.\ Matthews, and J.\ Purcell, $A$-polynomials, Ptolemy equations and Dehn filling (2021 preprint).
arXiv:2002.10356

\bibitem{HMPT} J.\ Howie, D.\ Matthews, J.\ Purcell, and E.\ Thompson, $A$-polynomials of filling of the Whitehead sister (2021 preprint).
 arXiv:2106.13462
 
 \bibitem{IMNS} M.\ Ishikawa, T.W.\ Mattman, K.\ Namiki, and K.W.\ Shimokawa,
Alternating knots with large boundary slope diameter.
Contemp. Math., \textbf{760} 
Centre Rech. Math. Proc.
American Mathematical Society, [Providence], RI, (2020), 207–216.

 \bibitem{JR}
 W.\ Jaco and H.\ Rubinstein, Layered triangulations of 3-manifolds (2006 preprint).
 arXiv:math/0603601
 
\bibitem{Kabaya}
Y. Kabaya,
Private communication,
April 2008

\bibitem{Kabaya-JKTR}
Y. Kabaya,
A method to find ideal points from ideal triangulations.
J.\ Knot Theory Ramifications  \textbf{19}  (2010), 509--524.
arXiv:0706.0971 

\bibitem{Matsuda} 
H.\ Matsuda, 
Boundary slopes of knots of braid index three,
preprint.

\bibitem{Sage}
SageMath, the Sage Mathematics Software System (Version 9.6),
   The Sage Developers, 2023, https://www.sagemath.org.
\end{thebibliography}
\end{document}